\documentclass[b5paper,reqno]{amsart}
\usepackage{amsmath,amssymb,amsfonts,amsthm,latexsym,amstext,amscd}
\usepackage{epsfig,graphics,graphicx,color}
\usepackage{epstopdf}
\usepackage{pdfsync}
\usepackage{enumerate}
\usepackage[applemac]{inputenc}
\usepackage{color}
\usepackage{hyperref}
\hypersetup{pdfborder={0 0 0},colorlinks,citecolor=red,filecolor=red,linkcolor=red,urlcolor=green}
\makeatletter \thm@headfont{\bfseries\scshape} \makeatother
\allowdisplaybreaks

\newtheorem{thm}{Theorem}
\newtheorem{lem}[thm]{Lemma}
\newtheorem{cor}[thm]{Corollary}

\newtheorem{rem}[thm]{Remark}

\begin{document}\parindent0pt

\title[Non-uniform distribution in the $3x+1$ orbits]{A non-uniform distribution property of most orbits, in case the $3x+1$ conjecture is true}

\author[A. THOMAS]
       {Alain THOMAS\textsuperscript{\dag{}}
       }
\thanks{\textsuperscript{\dag{}} This author is partially supported by Aix-Marseille University, I2M, Marseille, France}
\address{{\bf {Alain THOMAS}}\\
448 all\'ee des Cantons\\
83640 Plan d'Aups Sainte Baume\\
FRANCE
}
\email{alain-yves.thomas@laposte.net}



\begin{abstract}
Let $T(n)=\left\{\begin{array}{ll}3n+1&(n\hbox{ odd})\\\frac n2&(n\hbox{ even})\end{array}\right.$ ($n\in\mathbb Z$). We call ``the orbit of the integer $n$", the set
$$
\mathcal O_n:=\{m\in\mathbb Z\;:\;\exists k\ge0,\ m=T^k(n)\}
$$
and we put $c_i(n):=\#\{m\in\mathcal O_n\;:\;m\equiv i\hbox{ mod.}18\}$. Let $W$ be the set of the integers whose orbit contains $1$ and is, in the following sense, approximately well distributed modulo $18$ between the six elements of the set $I:=\{1,5,7,11,13,17\}$ (the elements of \{1,\dots,18\} that are odd and not divisible by $3$). More precisely:
$$
W:=\Big\{n\in\mathbb Z\;:\;\exists k\ge0,\ T^k(n)=1\hbox{ and }\forall i\in I,\ \frac{c_i(n)}{\sum_{i\in I}c_i(n)}\le\frac16+0.0215\Big\}.
$$
We prove that $W\cap\mathbb N$ has density $0$ in $\mathbb N$. Consequently, if the $3x+1$ conjecture is true, most of the positive integers $n$ satisfy
$$
\frac{\max_{i\in I}c_i(n)}{\sum_{i\in I}c_i(n)}>\frac16+0.0215.
$$
\end{abstract} 

\maketitle

\medskip

\begin{centerline}%
{\dedicatory\textsl{Dedicated to the memory of Pierre Liardet}}
\end{centerline}

\section{Introduction}\label{sec:one}

As it can bee seen in one example given by Lagarias, if we chose a large integer (for instance the one of figure 2 in \href{http://www.ams.org/bookstore/pspdf/mbk-78-prev.pdf}{http://www.ams.org/bookstore/pspdf/mbk-78-prev.pdf}), in general its orbit  under the transform $T:=n\mapsto\left\{\begin{array}{ll}3n+1&(n\hbox{ odd})\\\frac n2&(n\hbox{ even})\end{array}\right.$ contains about two times less odd numbers than even numbers, due to the fact that $3n+1$ is even for any odd $n$. This figure shows that the orbit of the integer $100\lfloor\pi\cdot10^{35}\rfloor$ has length about $900$ and, as expected, about $300$ odd and $600$ even elements, because $100\lfloor\pi\cdot10^{35}\rfloor\cdot\frac{3^{300}}{2^{600}}\approx1$. Fortunately, the method we use to prove the following theorem can't be used to contradict this property.

\begin{thm}\label{thetheor}
We put
$$
\mathcal O_n:=\{m\in\mathbb Z\;:\;\exists k\ge0,\ m=T^k(n)\}\quad(\hbox{orbit of the integer }n),
$$
$$
c_i(n):=\#\{m\in\mathcal O_n\;:\;m\equiv i\hbox{ mod.}18\}\quad(\hbox{finite or infinite}),
$$
$$
I:=\{1,5,7,11,13,17\},
$$
$$
W:=\Big\{n\in\mathbb Z\;:\;\exists k\ge0,\ T^k(n)=1\hbox{ and }\forall i\in I,\ \frac{c_i(n)}{\sum_{i\in I}c_i(n)}\le\frac16+0.0215\Big\}.
$$
We have for any $N$ large enough
$$
\#W\cap\{1,\dots,N\}\le N^{0.9999}.
$$
\end{thm}

Of course this theorem remains true if we replace the condition $T^k(n)=1$ by $T^k(n)=n_0$, where $n_0\in\mathbb Z\setminus\{0\}$ is fixed. In case $n_0<0$ we replace the interval $\{1,\dots,N\}$ by $\{-N,\dots,-1\}$.

To prove this theorem we use the same method as Krasilov and Lagarias \cite{Kr}, it consists in describing the set of the antecedents of $1$ by the powers of $T$. See also \cite{C,C2,Ko,Lbiblio,LbiblioII,Lbook,M}.

\begin{rem}To give a numerical example we consider the orbit of each of the integers $n\in\{1,\dots,26\}$ and we compute $c_i:=\sum_{n=1}^{26}c_i(n)$:
$$
(c_1,\dots,c_{18})=(28,41,5,49,22,4,5,37,2,23,10,2,11,4,1,47,13,1).
$$
As expected, $\sum_{i{\rm\ odd}}c_i=97$ is close to the half of $\sum_{i{\rm\ even}}c_i=208$. Among the $c_i$ with $i$ odd, $c_7=5$ is smaller than $c_1=28$, $c_5=22$, $c_{11}=10$, $c_{13}=11$ and $c_{17}=13$. The proof of the theorem allows to see, in the general case when $n\in\{1,\dots,N\}$, why $c_7$ is smaller than $c_1$, $c_5$, $c_{11}$, $c_{13}$, $c_{17}$. On the other hand the $c_i$ for $i$ a multiple of $3$ are small for an obvious reason: $\frac{3n+1}{2^k}$ is never a multiple of $3$.
\end{rem}

\section{The notations we use to describe the set of the antecedents of $1$}\label{sec:two}

Instead of $T$ we use the transform defined by Sinai in \cite{S}, that we call $S$:
$$
\begin{array}{l}S:\sqcap\to\sqcap\\
\sqcap:=\{n\in\mathbb Z\;:\;n\hbox{ odd and }n\not\in3\mathbb Z\}=\{1,5,7,11,13,17\}+18\mathbb Z\\
\displaystyle S(n):=\frac{3n+1}{2^k},\ k\in\mathbb N.\end{array}
$$
The antecedents of $1$ by $S$ are the integers
\begin{equation}\label{firstinteger}
n_1=\frac13(2^{\varepsilon_1}-1)
\end{equation}
that belong to $\sqcap$; this is equivalent to $\varepsilon_1\in\{2,4\}$ mod.$6$. Let now $n_\alpha$,$n_{\alpha-1}$, \dots, $n_1$ be some integers such that
$$
n_\alpha\ \xrightarrow{S}\ n_{\alpha-1}\ \xrightarrow{S}\ \dots\ \xrightarrow{S}\ n_1\ \xrightarrow{S}\ n_0:=1.
$$
For any $0\le j<\alpha$ there exists $\varepsilon_{j+1}\in\mathbb N$ such that
\begin{equation}\label{integer}
n_{j+1}=\frac13\left(2^{\varepsilon_{j+1}}n_j-1\right).
\end{equation}
One has $n_{j+1}\in\sqcap$, and this is equivalent to $2^{\varepsilon_{j+1}}n_j-1\in3\mathbb N\setminus9\mathbb N$. This means that, when we know the value of $n_j$, or equivalently when we know $\varepsilon_1,\dots,\varepsilon_j$, the positive integer $\varepsilon_{j+1}$ must satisfy the conditions:
\begin{equation}\label{cases}
\begin{array}{ll}
\hbox{if }n_j\equiv 1\hbox{ mod.}18,&\varepsilon_{j+1}\in\{2,4\}\hbox{ mod.}6\\
\hbox{if }n_j\equiv 5\hbox{ mod.}18,&\varepsilon_{j+1}\in\{3,5\}\hbox{ mod.}6\\
\hbox{if }n_j\equiv 7\hbox{ mod.}18,&\varepsilon_{j+1}\in\{4,6\}\hbox{ mod.}6\\
\hbox{if }n_j\equiv 11\hbox{ mod.}18,&\varepsilon_{j+1}\in\{1,3\}\hbox{ mod.}6\\
\hbox{if }n_j\equiv 13\hbox{ mod.}18,&\varepsilon_{j+1}\in\{2,6\}\hbox{ mod.}6\\
\hbox{if }n_j\equiv 17\hbox{ mod.}18,&\varepsilon_{j+1}\in\{1,5\}\hbox{ mod.}6
\end{array}
\end{equation}
(notice that the case $n_j\equiv 7\hbox{ mod.}18$ gives the largest values: $\varepsilon_{j+1}\ge4$ and $n_{j+1}\ge\frac13(16n_j-1)$). So all the antecedents of $1$ by $S^\alpha$ are obtained by the following formula, subject to the conditions (\ref{cases}):
\begin{equation}\label{n}
n_\alpha=\frac1{3^\alpha}\left(2^{\varepsilon_1+\dots+\varepsilon_\alpha}-2^{\varepsilon_2+\dots+\varepsilon_\alpha}3^0-\dots-2^{\varepsilon_\alpha}3^{\alpha-2}-2^03^{\alpha-1}\right).
\end{equation}

We give a first estimation of $n_\alpha$:

\begin{lem}\label{notmain}
If $\alpha\ge2$ and $\varepsilon_1\ne2$,
$$
\frac{2^{\varepsilon_1+\dots+\varepsilon_\alpha}}{\alpha3^\alpha}\le n_\alpha\le\frac{2^{\varepsilon_1+\dots+\varepsilon_\alpha}}{3^\alpha}.
$$
\end{lem}

\begin{proof}
The upper bound is an immediate consequence of (\ref{n}). The lower bound can be deduced from the straightforward equality:
\begin{equation}\label{equality}
3n+1=3^{1+\alpha_n}n\quad\hbox{with}\quad\alpha_n=\frac1{\log3}\log\Big(1+\frac1{3n}\Big)\le\frac1{3n}.
\end{equation}
Indeed (\ref{integer}) and (\ref{equality}) imply
$$
n_j\le\frac{3^{1+\frac1{3n_{j+1}}}}{2^{\varepsilon_{j+1}}}n_{j+1}
$$
hence
$$
n_0\le\frac{3^{\alpha+\frac1{3n_1}+\dots+\frac1{3n_\alpha}}}{2^{\varepsilon_1+\dots+\varepsilon_\alpha}}n_\alpha.
$$
Now the $n_j$ are distinct (no cycle between $n_\alpha$ and $1$, because (\ref{firstinteger}) and the hypothesis $\varepsilon_1\ne2$ imply $n_1\ne1$), and consequently $\frac1{n_1}+\dots+\frac1{n_\alpha}\le\frac11+\dots+\frac1\alpha\le1+\log\alpha$. If $\alpha\ge2$, the inequality $3^{\frac1{3n_1}+\dots+\frac1{3n_\alpha}}\le\alpha$ and the lemma follow.
\end{proof}

Here we give a indexation and a new lower bound for of $n_\alpha$:

\begin{lem}\label{main}

There exists a one-to-one map
$$
\frak n:(\mathbb N\setminus\{1\})\times\mathbb N^{\alpha-1}\leftrightarrow\sqcap_\alpha:=\{n\in\sqcap\;:\;S^\alpha(n)=1\ne S^{\alpha-1}(n)\}
$$
such that -- for any $(i_1,\dots,i_\alpha)\in\mathbb N\setminus\{1\})\times\mathbb N^{\alpha-1}$
\begin{equation}\label{lb}
\begin{array}{l}\displaystyle\frak n(i_1,\dots,i_\alpha)\ge\frac{2^{3(i_1+\dots+i_\alpha)-c(\frak n(i_1,\dots,i_{\alpha-1}))+\alpha'(i_1,\dots,i_{\alpha})}}{\alpha3^\alpha},\hbox{ where}\\
c(n):=2c_1(n)+c_5(n)+3c_{11}(n)+2c_{13}(n)+3c_{17}(n)\quad(n\in\mathbb Z),\\\alpha'(i_1,\dots,i_{\alpha}):=\#\{1\le j\le\alpha\;:\;i_j\hbox{ odd}\}.\end{array}
\end{equation}
\end{lem}

\begin{proof}We define $\frak n(i_1,\dots,i_\alpha)$ by induction on $\alpha$. When $\alpha=1$, according to (\ref{firstinteger}) the antecedents of $1$ by $S$, distinct from $1$, are the following integers indexed by $i_1\ge2$:
\begin{equation}\label{firstcase}
\frak n(i_1):=\frac13\left(2^{\varepsilon_1(i_1)}-1\right)\hbox{ where }\forall i,\ \varepsilon_1(i):=3i-1\ (i\hbox{ odd) or }3i-2\ (i\hbox{ even))}.
\end{equation}
Suppose now that $\frak n(i_1,\dots,i_j)$ (antecedent of $1$ by $S^j$ and not by $S^{j-1}$) is already defined for any $(i_1,\dots,i_j)\in(\mathbb N\setminus\{1\})\times\mathbb N^{j-1}$. We denote by $0<\varepsilon(i_1,\dots,i_j,1)<\varepsilon(i_1,\dots,i_j,2)<\dots$ ($i\in\mathbb N$) the possible values of $\varepsilon_{j+1}$ in (\ref{cases}); the antecedents of $\frak n(i_1,\dots,i_j)$ by $S$ are
$$
\frak n(i_1,\dots,i_j,i):=\frac13\left(2^{\varepsilon(i_1,\dots,i_j,i)}\frak n(i_1,\dots,i_j)-1\right)\quad(i\in\mathbb N).
$$
We obtain in this way all the antecedents of $1$ by $S^{j+1}$ that are not antecedents of~$1$ by $S^j$. With this notations, the conditions in (\ref{cases}) are equivalent to
$$
\begin{array}{llllll}
\hbox{if }\frak n(i_1,\dots,i_j)\equiv 1,&\varepsilon(i_1,\dots,i_j,i)=3i-1&(i\hbox{ odd})&\hbox{or}&3i-2&(i\hbox{ even})\\
\hbox{if }\frak n(i_1,\dots,i_j)\equiv 5,&\varepsilon(i_1,\dots,i_j,i)=3i&(i\hbox{ odd})&\hbox{or}&3i-1&(i\hbox{ even})\\
\hbox{if }\frak n(i_1,\dots,i_j)\equiv 7,&\varepsilon(i_1,\dots,i_j,i)=3i+1&(i\hbox{ odd})&\hbox{or}&3i&(i\hbox{ even})\\
\hbox{if }\frak n(i_1,\dots,i_j)\equiv 11,&\varepsilon(i_1,\dots,i_j,i)=3i-2&(i\hbox{ odd})&\hbox{or}&3i-3&(i\hbox{ even})\\
\hbox{if }\frak n(i_1,\dots,i_j)\equiv 13,&\varepsilon(i_1,\dots,i_j,i)=3i-1&(i\hbox{ odd})&\hbox{or}&3i&(i\hbox{ even})\\
\hbox{if }\frak n(i_1,\dots,i_j)\equiv 17,&\varepsilon(i_1,\dots,i_j,i)=3i-2&(i\hbox{ odd})&\hbox{or}&3i-1&(i\hbox{ even}).\end{array}
$$
Setting $r(i):=\left\{\begin{array}{ll}1&(i\hbox{ odd})\\0&(i\hbox{ even})\end{array}\right.$ (remainder of $n$ modulo $2$), we have for any $i\in\mathbb N$
\begin{equation}\label{newcases}
\begin{array}{llll}
\hbox{if }\frak n(i_1,\dots,i_j)\equiv 1,&\varepsilon(i_1,\dots,i_j,i)\ge3i-2+r(i)\\
\hbox{if }\frak n(i_1,\dots,i_j)\equiv 5,&\varepsilon(i_1,\dots,i_j,i)\ge3i-1+r(i)\\
\hbox{if }\frak n(i_1,\dots,i_j)\equiv 7,&\varepsilon(i_1,\dots,i_j,i)\ge3i-0+r(i)\\
\hbox{if }\frak n(i_1,\dots,i_j)\equiv 11,&\varepsilon(i_1,\dots,i_j,i)\ge3i-3+r(i)\\
\hbox{if }\frak n(i_1,\dots,i_j)\equiv 13,&\varepsilon(i_1,\dots,i_j,i)\ge3i-2+r(i)\\
\hbox{if }\frak n(i_1,\dots,i_j)\equiv 17,&\varepsilon(i_1,\dots,i_j,i)\ge3i-3+r(i).\end{array}
\end{equation}
We consider the formula (\ref{firstcase}), and the formulas (\ref{newcases}) for $j=1,\dots,\alpha-1$: they depend on the value modulo $18$ of the integers $1,\frak n(i_1),\dots,\frak n(i_1,\dots,i_{\alpha-1})$ respectively. In other words, these formulas depend on the orbit of $\frak n(i_1,\dots,i_{\alpha-1})$ by $S$. Using Lemma \ref{notmain} and the definitions of $c(n)$ and $\alpha'(i_1,\dots,i_{\alpha})=\sum_{j=0}^\alpha r(i_j)$ we deduce
$$
\frak n(i_1,\dots,i_\alpha)\ge\frac{2^{\sum_{j=1}^\alpha\varepsilon(i_1,\dots,i_j)}}{\alpha3^\alpha}\ge\frac{2^{3(i_1+\dots+i_\alpha)-c(\frak n(i_1,\dots,i_{\alpha-1}))+\alpha'(i_1,\dots,i_{\alpha})}}{\alpha3^\alpha}.
$$
\end{proof}

\section{A first bound for $\#W\cap\{1,\dots,N\}$}\label{sec:three}

In the following lemma we specify how to obtain all the antecedents of $1$ by the powers of $S$ or~$T$.

\begin{lem}\label{ant}
(i) The set of the antecedents of $1$ by the powers of $S$ (resp. by the powers of $T$), namely
$$
\mathcal S:=\{n\in\sqcap\;:\;\exists\alpha\ge0,\ S^\alpha(n)=1\}\quad(\hbox{resp. }\mathcal T:=\{n\in\mathbb N\;:\;\exists k\ge0,\ T^k(n)=1\}),
$$
can also be defined by
$$
\mathcal S=\bigcup_{\alpha\ge0}\sqcap_\alpha\ (\hbox{where}\ \sqcap_0:=\{1\})\quad\hbox{and}\quad\mathcal T=\mathbb N\cap\bigcup_{i\ge0}\bigcup_{j\ge1}\frac{2^i}3(2^j\mathcal S-1).
$$
(ii) If $n\in\mathcal S$ there exist $\alpha\ge0,i'_1,\dots,i'_\alpha\ge1$ and $A\subset\{1,\dots,\alpha\}$ such that
\begin{equation}\label{iprime}
n=\frak n(i_1,\dots,i_\alpha)\hbox{ with }i_j=\left\{\begin{array}{ll}2i'_j-1&\hbox{if }j\in A\\2i'_j&\hbox{else}\end{array}\right.
\end{equation}
and $i'_1\ne1$if $1\in A$. If $n=\frak n(i_1,\dots,i_\alpha)$ belongs to $W$,
\begin{equation}\label{ineq}
i'_1+\dots+i'_\alpha-\frac{11}6\Big(\frac16+0.0215\Big)(\alpha+1)-\frac13\#A-\frac{\log\alpha}{6\log2}-\alpha\frac{\log3}{6\log2}\le\frac{\log n}{6\log2}.
\end{equation}
\end{lem}

\begin{proof}
(i) The first relation is obvious and the second follows from the fact that the orbit of any $n\in\mathbb N$ by the transformation $T$, begins by
$$
n\ \xrightarrow{T^i\ (i\ge0)}\ \frac n{2^i}\ \xrightarrow{T}\ 3\frac n{2^i}+1\ \xrightarrow{T^j\ (j\ge1)}\ \frac{3\frac n{2^i}+1}{2^j}\ \in\sqcap\quad(\hbox{or }\in\mathcal S,\hbox{ if }n\in\mathcal T).
$$

(ii) (\ref{iprime}) is a consequence of Lemma \ref{main}. For any $n\in W$,
\begin{equation}\label{c(n)}
c(n)\le11\Big(\frac16+0.0215\Big)\sum_{i\in I}c_i(n)=11\Big(\frac16+0.0215\Big)\#\{m\in\mathcal O_n\;:\;m\hbox{ odd}\}.
\end{equation}
Now if n$=\frak n(i_1,\dots,i_\alpha)$ there are $\alpha+1$ odd integers in $\mathcal O_n$, so (\ref{ineq}) follows from (\ref{lb}) and (\ref{c(n)}).
\end{proof}

\begin{lem}\label{upper}
There exists a constant $K$ such that
$$
\#W\cap\{1,\dots,N\}\le K(\log N)^K\max_{(\alpha',\alpha'')\in A(N)}\frac{(\alpha'+\alpha''+\alpha''')!}{\alpha'!\ \alpha''!\ \alpha'''!}
$$
where $\alpha'''=\alpha'''(N,\alpha',\alpha'')$ is defined by
$$
\alpha'''(N,\alpha',\alpha''):=\lfloor0.2405\log N+0.345-0.05749\ \alpha'-0.39083\ \alpha''\rfloor
$$
and $A(N):=\{(\alpha',\alpha'')\in(\mathbb N\cup\{0\})^2\;:\;\alpha'''(N,\alpha',\alpha'')\ge0\}$.
\end{lem}

\begin{proof}
We use the one-to-one map $\frak n:(\mathbb N\setminus\{1\})\times\mathbb N^{\alpha-1}\leftrightarrow\sqcap_\alpha$ defined in Lemma~\ref{main}, and the notation
$$
\begin{array}{l}A(i_1,\dots,i_\alpha):=\{1\le j\le\alpha\;:\;i_j\hbox{ odd}\}\\
\sqcap_{\alpha,\alpha'}:=\{n=\frak n(i_1,\dots,i_\alpha)\in\sqcap_\alpha\;:\;\#A(i_1,\dots,i_\alpha)=\alpha'\}.\end{array}
$$
Now the nonnegative integers $N\in\mathbb N$, $\alpha,\alpha',\alpha''\ge0$ are fixed with $\alpha=\alpha'+\alpha''$. Assume for instance that $\alpha\ge10^{10}$: then we have $\frac{\log\alpha}\alpha\le10^{-8}$. We use the notations of Lemma \ref{ant} (ii); we deduce from~(\ref{ineq}) that, if there exists at least one element $n=\frak n(i_1,\dots,i_\alpha)\in\sqcap_{\alpha,\alpha'}\cap W\cap\{1,\dots,N\}$,
\begin{equation}\label{deduced}
i'_1+\dots+i'_\alpha-0.345(\alpha+1)-0.33334\ \alpha'-0.26417\ \alpha\le0.2405\log n.
\end{equation}
This inequality is equivalent to
$$
i'_1+\dots+i'_\alpha\le\alpha+\alpha'''(n,\alpha',\alpha'').
$$
This last inequality with $\alpha\le i'_1+\dots+i'_\alpha$ implies $\alpha'''(n,\alpha',\alpha'')\ge0$ and a fortiori $\alpha'''(N,\alpha',\alpha'')\ge0$. So we have proved that, if the set $\sqcap_{\alpha,\alpha'}\cap W\cap\{1,\dots,N\}$ is not empty, $(\alpha',\alpha'')$ belongs to $A(N)$. 

Let us bound the number of elements of $\sqcap_{\alpha,\alpha'}\cap W\cap\{1,\dots,N\}$. We can associate injectively to any $n$ in this set, some integers $i'_1,\dots,i'_\alpha$ such that
$$
1\le i'_1<i'_1+i'_2<\dots<i'_1+\dots+i'_\alpha\le\alpha+\alpha'''\quad(\hbox{where }\alpha'''=\alpha'''(N,\alpha',\alpha''))
$$
and a subset $A\subset\{1,\dots,\alpha\}$ of cardinality $\alpha'$, such that
$$
A(i_1,\dots,i_\alpha)=A,\hbox{ where }i_j=\left\{\begin{array}{ll}2i'_j-1&\hbox{if }n\in A\\2i'_j&\hbox{else}.\end{array}\right.
$$
Consequently
$$
\#\sqcap_{\alpha,\alpha'}\cap W\cap\{1,\dots,N\}\le\left(\begin{array}{c}\alpha+\alpha'''\\\alpha\end{array}\right)\cdot\left(\begin{array}{c}\alpha\\\alpha'\end{array}\right)=\frac{(\alpha'+\alpha''+\alpha''')!}{\alpha'!\ \alpha''!\ \alpha'''!}.
$$
The inequality $\alpha'''(N,\alpha',\alpha'')\ge0$ implies $\alpha\le K_1\log N$ with $K_1$ constant, hence
\begin{equation}\label{firstsum}
\sum_{\alpha=10^{10}}^{\lfloor K_1\log N\rfloor}\sum_{\alpha'=0}^\alpha\#\sqcap_{\alpha,\alpha'}\cap W\cap\{1,\dots,N\}\le(K_1\log N)^2\max_{(\alpha',\alpha'')\in A(N)}\frac{(\alpha'+\alpha''+\alpha''')!}{\alpha'!\ \alpha''!\ \alpha'''!}.
\end{equation}
It remains to bound $\displaystyle\sum_{\alpha=0}^{10^{10}-1}\#\sqcap_{\alpha}\cap W\cap\{1,\dots,N\}$. One can associate injectively to any $n\in\sqcap_{\alpha}\cap W\cap\{1,\dots,N\}$, some positive integers $\varepsilon_1,\dots,\varepsilon_\alpha$ such that (\ref{n}) holds. According to Lemma \ref{notmain} one has $2^{\varepsilon_1+\dots+\varepsilon_\alpha}\le\alpha3^\alpha N<10^{10}3^{10^{10}}N$, hence any $\varepsilon_j$ is bounded by $K_2\log N$ with $K_2$ constant. Consequently
\begin{equation}\label{secondsum}
\sum_{\alpha=0}^{10^{10}-1}\#\sqcap_{\alpha}\cap W\cap\{1,\dots,N\}\le10^{10}(K_2\log N)^{10^{10}}.
\end{equation}
From Lemma \ref{ant} (i), $\mathcal S$ is the union of the $\sqcap_{\alpha}$ hence, from (\ref{firstsum}) and (\ref{secondsum}), there exists a constant $K_3$ such that
\begin{equation}\label{forS}
\#\mathcal S\cap W\cap\{1,\dots,N\}\le K_3(\log N)^{K_3}\max_{(\alpha',\alpha'')\in A(N)}\frac{(\alpha'+\alpha''+\alpha''')!}{\alpha'!\ \alpha''!\ \alpha'''!}.
\end{equation}
Let us bound now $\#W\cap\{1,\dots,N\}$. By Lemma \ref{ant} (i) any $n\in\mathcal T\cap W\cap\{1,\dots,N\}=W\cap\{1,\dots,N\}$ can be written
$$
n=\frac{2^i}3(2^js-1)\hbox{ with }i\ge0,\ j\ge1,\ s\in\mathcal S\cap W.
$$
This implies $s\le2N$, $2^i\le3N$ and $2^j\le3N+1$. So there are at most $K_4(\log N)^2$ possible values for the couple $(i,j)$, with $K_4$ constant, and
\begin{equation}\label{forT}
\#W\cap\{1,\dots,N\}\le K_4(\log N)^2\#\mathcal S\cap W\cap\{1,\dots,2N\}.
\end{equation}
The lemma follows from (\ref{forS}) and (\ref{forT}).
\end{proof}

\section{Proof of the theorem}

\begin{lem}\label{upperbis}
Let $\ell(N):=0.2405\log N+0.345$. With the notations of Lemma \ref{upper}, one has
$$
\max_{(\alpha',\alpha'')\in A(N)}\frac{(\alpha'+\alpha''+\alpha''')!}{\alpha'!\ \alpha''!\ \alpha'''!}\le\max_{(x,y)\in T}\left(\frac{(x+y+z)^{x+y+z}}{x^x\ y^y\ z^z}\right)^{\ell(N)}
$$
where $z=1-0.05749\ x-0.39083\ y$ and $T:=\{(x,y)\;:\;x\ge0,y\ge0,z\ge0\}$.
\end{lem}

\begin{proof}
Let $(\alpha',\alpha'')\in A(N)$, the reals
$$
x=\frac{\alpha'}{\ell(N)}\quad\hbox{and}\quad y=\frac{\alpha''}{\ell(N)},
$$
satisfy $(x,y)\in T$. By Corollary \ref{classical} one has

\begin{equation}\label{tierce}
\frac{(\alpha'+\alpha''+\alpha''')!}{\alpha'!\ \alpha''!\ \alpha'''!}\le\frac{(\alpha'+\alpha''+\alpha''')^{\alpha'+\alpha''+\alpha'''}}{\alpha'^{\alpha'}\alpha''^{\alpha''}\alpha'''^{\alpha'''}}.
\end{equation}

The map $t\mapsto\frac{(\alpha'+\alpha''+t)^{\alpha'+\alpha''+t}}{\alpha'^{\alpha'}\alpha''^{\alpha''}t^t}$ is not decreasing because the derivative of its logarithm is $\log(1+\frac{\alpha'+\alpha''}t)\ge0$. Recall that $\alpha'''=\alpha'''(N,\alpha',\alpha'')$ is the integral part~of
$$
\alpha''''=\alpha''''(N,\alpha',\alpha''):=\ell(N)-0.0575\ \alpha'-0.39084\ \alpha'',
$$
so one has
\begin{equation}\label{quarte}
\frac{(\alpha'+\alpha''+\alpha''')^{\alpha'+\alpha''+\alpha'''}}{\alpha'^{\alpha'}\ \alpha''^{\alpha''}\ \alpha'''^{\alpha'''}}\le\frac{(\alpha'+\alpha''+\alpha'''')^{\alpha'+\alpha''+\alpha''''}}{\alpha'^{\alpha'}\ \alpha''^{\alpha''}\ \alpha''''^{\alpha''''}}.
\end{equation}
Lemma \ref{upperbis} results from (\ref{tierce}) and (\ref{quarte}) because the real $z$, as defined in this lemma, is equal to $\frac{\alpha''''}{\ell(N)}$.
\end{proof}

\begin{proof}[End of the proof of the theorem]
We apply Lemma \ref{interior} to $a=0.05749$ and $b=0.39083$: the function $\varphi$ attains its maximum in the interior of $T$, let $(x_0,y_0)$ be a point where $\varphi$ is maximal and let $z_0=1-ax_0-by_0$. Since $\varphi$ is differentiable on the interior of $T$, the partial derivatives are null at $(x_0,y_0)$:
$$
\left\{\begin{array}{lcl}(1-a)\log(x_0+y_0+z_0)-\log x_0+a\log z_0&=&0\\
(1-b)\log(x_0+y_0+z_0)-\log y_0+b\log z_0&=&0.\end{array}\right.
$$
Let $w_0:=0.2405\ \varphi(x_0,y_0,z_0)$; we compute $x_0,y_0,z_0,w_0$ by approximation and obtain
$$
w_0-0.9998\in(0,10^{-4}).
$$
From Lemma \ref{upper} and Lemma \ref{upperbis},
$$
\begin{array}{rcl}\#W\cap\{1,\dots,N\}&\le&\hbox{constant}\cdot(\log N)^K\cdot e^{\varphi(x_0,y_0,z_0)\ 0.2405\log N}\\
&\le&\hbox{constant}\cdot(\log N)^KN^{w_0}\end{array}
$$
hence $\#W\cap\{1,\dots,N\}\le N^{0.9999}$ for $N$ large enough.\end{proof}

\appendix

\section{Classical bound for the binomial coefficients}

\begin{lem}\label{bin}
For any $m,n\in\mathbb N$, $\frac{(m+n)!}{m!n!}\le\frac{(m+n)^{m+n}}{m^mn^n}$.
\end{lem}

\begin{proof}
Let $f(m,n)=\frac{(m+n)!}{m!n!}\cdot\frac{m^mn^n}{(m+n)^{m+n}}$, one has obviously $f(m,1)\le1$ and it remains to prove that $f(m,n+1)\le f(m,n)$.

Let $r=n+1$ and $s=m+n+1$, one has
$$
\begin{array}{rcl}\frac{f(m,n+1)}{f(m,n)}
&=&\frac{(m+n+1)!}{m!(n+1)!}\cdot\frac{m^m(n+1)^{n+1}}{(m+n+1)^{m+n+1}}\cdot\frac{m!n!}{(m+n)!}\cdot\frac{(m+n)^{m+n}}{m^mn^n}\\
&=&\frac{(n+1)^{n}}{(m+n+1)^{m+n}}\cdot\frac{(m+n)^{m+n}}{n^n}\\
&=&\big(\frac r{r-1}\big)^{r-1}\cdot\big(\frac {s-1}s\big)^{s-1}.\end{array}
$$
Since $r<s$ it remains to prove that the function $g(x)=\big(\frac x{x-1}\big)^{x-1}$ is increasing: this holds because
$$
\begin{array}{rcl}(\log g(x))'&=&\log\big(\frac x{x-1}\big)-\frac 1x\\
&=&-\log(1+t)+t\quad\hbox{with }t=-\frac 1x\\
&\ge&0.\end{array}
$$
\end{proof}

\begin{cor}\label{classical}
For any $m,n,p\in\mathbb N$, $\frac{(m+n+p)!}{m!n!p!}\le\frac{(m+n+p)^{m+n+p}}{m^mn^np^p}$.
\end{cor}

\begin{proof}
Using Lemma \ref{bin},
$$
\begin{array}{rcl}\frac{(m+n+p)!}{m!n!p!}&=&\frac{(m+n+p)!}{(m+n)!p!}\frac{(m+n)!}{m!n!}\\
&\le&\frac{(m+n+p)^{m+n+p}}{(m+n)^{m+n}p^p}\frac{(m+n)^{m+n}}{m^mn^n}=\frac{(m+n+p)^{m+n+p}}{m^mn^np^p}.\end{array}
$$
\end{proof}

\section{Study of a function}

\begin{lem}\label{interior}
Let $a,b\in(0,1]$. The maximum of the function
$$
\varphi(x,y):=\log\Big(\frac{(x+y+z)^{x+y+z}}{x^xy^yz^z}\Big)\quad(\hbox{where }z=1-ax-by)
$$
is attained in the interior of the triangle $T:=\{(x,y)\;:\;x\ge0,y\ge0,z\ge0\}$.
\end{lem}

\begin{proof}$\varphi(x,y)=(x+y+z)\log(x+y+z)-x\log x-y\log y-z\log z$ is continuous on the closed triangle $T$ whose vertices are the origin, the point $A(\frac1a,0)$ and the point $B(0,\frac1b)$, hence it has a maximum on $T$.

\includegraphics[scale=0.3]{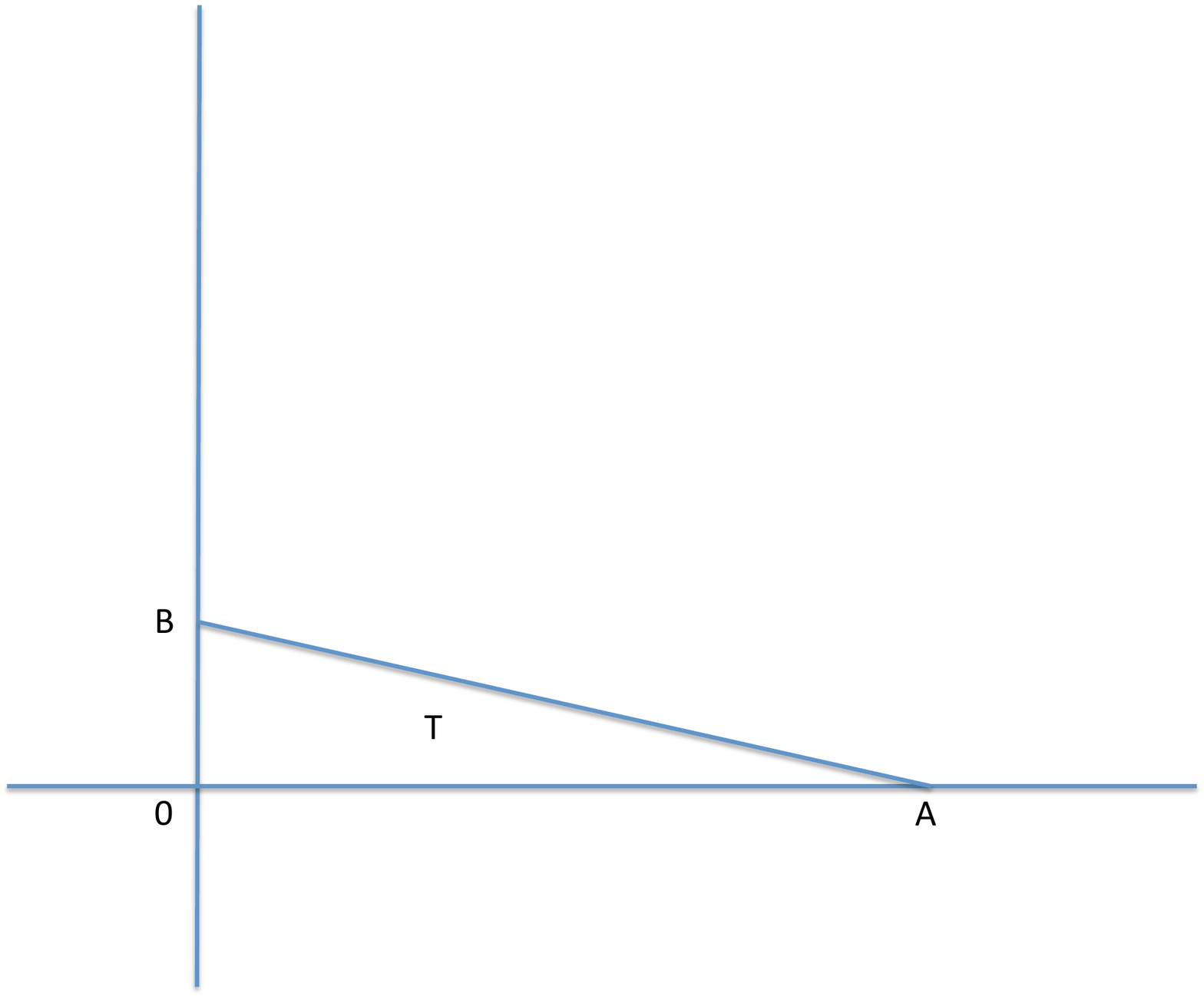}

Notice that $x+y+z\ne0$ on $T$. The partial derivative
$$
\frac{\partial}{\partial x}(\varphi(x,y))=(1-a)\log(x+y+z)-\log x+a\log z
$$
has limits $+\infty$ when $x\to0$ and $-\infty$ when $x\to x_1$, with $x_1$ such that $1-ax_1-by=0$, hence the map $x\mapsto\varphi(x,y)$ increases at the neighborough of $0$ and decreases at the neighborought of $x_1$, it has a maximum in the open interval $(0,x_1)$.

Similarly the map $y\mapsto\varphi(x,y)$ has a maximum in the open interval $(0,y_1)$, with $y_1$ such that $1-ax-by_1=0$.

We deduce that the map $(x,y)\mapsto\varphi(x,y)$ cannot have a maximum in the boundary of $T$.
\end{proof}

\end{document}